\documentclass[12pt,reqno]{amsart}
\usepackage{amsmath, amsthm, amssymb}

\usepackage{comment}
\usepackage{color}
\usepackage{url}
\usepackage{amscd}

\advance \topmargin by -\headheight
\advance \topmargin by -\headsep
     
\evensidemargin \oddsidemargin
\newtheorem{theorem}{Theorem}[section]
\newtheorem{lemma}[theorem]{Lemma}
\newtheorem{corollary}[theorem]{Corollary}

\newtheorem{proposition}[theorem]{Proposition}

\theoremstyle{definition}

\theoremstyle{remark}
\newtheorem{remark}[theorem]{Remark}

\numberwithin{equation}{section}

\def\bfb{{\mathbf b}}
\def\bfc{{\mathbf c}}

\def\bft{{\mathbf t}}
\def\bfu{{\mathbf u}}
\def\bfv{{\mathbf v}}

\def\bfx{{\mathbf x}}

\def\A{{\mathbb A}}
\def\P{{\mathbb P}}

\def\Z{{\mathbb Z}}\def\Q{{\mathbb Q}}

\def\grm{{\mathfrak m}}
\def\grn{{\mathfrak n}}

\def\grp{{\mathfrak p}}
\def\grP{{\mathfrak P}}

\def\grp{{\mathfrak p}}

\def\alp{{\alpha}} \def\bfalp{{\boldsymbol \alpha}}

\def\kap{{\kappa}}
\def\lam{{\lambda}}

\def\eps{\varepsilon}

\def\Pic{{\rm Pic}}

\def\Spec{{\rm Spec}}

\def\int{{\rm int}}

\def\Id{{\rm Id}}

\def\Qbar{{\overline{\Q}}}

\newenvironment{blue}{\color{blue}}{}

\specialcomment{comblue}{\begin{blue}}{\end{blue}}
\excludecomment{comblue}

\DeclareMathOperator{\Br}{Br}
\DeclareMathOperator{\Gal}{Gal}
\DeclareMathOperator{\Disc}{Disc}

\begin{document}

\renewcommand{\thefootnote}{\fnsymbol{footnote}}
\author[J\"org Jahnel]{J\"org Jahnel}

\address{\mbox{Department Mathematik\\ \!Univ.\ \!Siegen\\ \!Walter-Flex-Str.\ \!3\\ \!D-57068 \!Siegen\\ \!Germany}}
\email{jahnel@mathematik.uni-siegen.de}
\urladdr{http://www.uni-math.gwdg.de/jahnel}

\author[Damaris Schindler]{Damaris Schindler${}^{*}$}

\address{Mathematisch Instituut\\ \!Universiteit \!Utrecht\\ \!Budapestlaan~6\\ \!NL-3584 \!CD \!Utrecht\\ The Netherlands}
\email{d.schindler@uu.nl}
\urladdr{http://www.uu.nl/staff/DSchindler}

\title[On the frequency of algebraic Brauer classes]{On the frequency of algebraic Brauer classes on certain log K3 surfaces}

\date{\today}

\maketitle


\footnotetext[1]{The second author was supported by the {\em NWO\/} Veni Grant No.\ 016.Veni.173.016.}
\footnotetext[2]{Mathematics Subject Classification 2010, {Primary 11E12; Secondary 14G20, 14G25, 14J20}}

\begin{abstract}
Given systems of two (inhomogeneous) quadratic equations in four variables, it is known that the Hasse principle for integral points may fail. Sometimes this failure can be explained by some integral Brauer-Manin obstruction. We study the existence of a non-trivial algebraic part of the Brauer group for a family of such systems and show that the failure of the integral Hasse principle due to an algebraic Brauer-Manin obstruction is rare, as for a generic choice of a system the algebraic part of the Brauer-group is trivial. We use resolvent constructions to give quantitative upper bounds on the number of exceptions. 
\end{abstract}

\section{Introduction}

We consider a system of two rational quadratic forms in $5$ variables given by
$$Q_1(\bfx)=\bfx^t A\bfx,\quad Q_2(\bfx) =\bfx^tB\bfx,$$
where $A,B\in M_{5\times 5}(\Q)$ are two symmetric matrices with rational entries. For a generic choice of $A$ and $B$, the intersection 
\begin{equation}\label{eqn1}
S_{A,B}:\quad Q_1(\bfx)=Q_2(\bfx)=0
\end{equation}
defines a del Pezzo surface of degree four in $\P_{\Q}^4$. It is well known that the Hasse principle and weak approximation may fail for such surfaces (see for example \cite{BSD75}). In all known examples, for del Pezzo surfaces of degree four, the failure of the Hasse principle can be explained by a Brauer-Manin obstruction. I.e., in all of these situations one has adelic points on $S_{A,B}$ but the Brauer-Manin set $S_{A,B}(\A_\Q)^{\Br}$, that is the subset of adelic points that are in the kernel of the Brauer-Manin pairing with the Brauer group $\Br(S_{A,B})$, is empty.\par
More recently, Colliot-Th\'el\`ene and Xu \cite{CX09}, initiated the study of integral Brauer-Manin obstructions. In \cite{CX09}, they studied integral points on homogeneous spaces and representation problems by integral quadratic forms. In another direction, the concept of Brauer-Manin obstructions for affine varieties has been pursued in \cite{CTWit12} for families of affine cubic surfaces, such as representation problems of an integer by sums of three cubes. Moreover, see work of Bright and Lyczak \cite{BLA17} for certain log K3 surfaces and Berg \cite{BerA17} on the description of the Brauer-Manin obstruction for affine Ch{\^a}telet surfaces.\par
In the situation of a del Pezzo surface of degree four, one may choose a hyperplane $H\subset \P_{\Q}^4$ and consider the complement ${U:=S_{A,B}\setminus H}$. Integral points on an integral model $\mathcal{U}$ of $U$ then essentially correspond to integer solutions of systems of two inhomogeneous quadratic equations in four variables.  When does such a system have an integer solution? A necessary condition is that it has real solutions and solutions in $\Z_p$ for every prime $p$. However, these conditions may not be sufficient. Again, there could be a Brauer-Manin obstruction leading to a violation of the local-global principle. We are interested in the question how often one should expect a violation of the local-global principle due to a Brauer-Manin obstruction. A Brauer-Manin obstruction can only occur if there are non-constant Brauer classes in the Brauer group of the scheme $U$. The Brauer classes which vanish after a base change to the algebraic closure of $\Q$ are most accessible to computations. In this article we study upper bounds on how often the algebraic part of the Brauer-group is non-trivial for certain families of systems of two inhomogeneous quadratic equations in four variables. It would be very interesting to find lower bounds as well (for certain families see \cite{JS}), or get a prediction of the density of such examples within a given family.
\par
In \cite{JS}, the authors computed a list of the algebraic parts of the Brauer group that can occur for such surfaces assuming that the intersection $S_{A,B}\cap H$ is geometrically integral. Note that these results were recently extended to del Pezzo surfaces of degree at most $7$ by M. Bright and J. Lyczak \cite{BLA17}. We recall that the list in the case of the complement of a hyperplane section in a del Pezzo surface of degree four consists of $0,\Z/2\Z$, $(\Z/2\Z)^2$, $(\Z/2\Z)^3$, $(\Z/2\Z)^4$, $\Z/4\Z$, $\Z/2\Z\times \Z/4\Z$ and $(\Z/2\Z)^2\times \Z/4\Z$. Under the assumption that $H\cap S_{A,B}$ is geometrically irreducible (for a more general criterion see \cite[Lemma 4.1]{JS}), the algebraic part of the Brauer group is given by the first Galois cohomology group
$$\Br_1(U)/\Br_0(U)\cong H^1(\Gal (\bar\Q/\Q),\Pic (U_{\bar\Q})).$$
Moreover, any hyperplane $H$ lies in the anticanonical class and hence the algebraic part of the Brauer group is independent of which rational hyperplane we remove as long as we assume $H\cap S_{A,B}$ to be geometrically irreducible.\par
In this note, we study the question how often one typically expects to find a non-trivial algebraic part of the Brauer group $\Br_1(U)/\Br_0(U)$, where we vary over certain subfamilies of $S_{A,B}$. We fix a matrix $A\in M_{5\times 5}(\Z)$ with $\det (A)\neq 0$ and vary over integral matrices $B$. Note that if the intersection in (\ref{eqn1}) is smooth of codimension 2 as a projective scheme in $\P^4$ then $S_{A,B}$ is indeed a del Pezzo surface of degree four.\par
We define $N_{2}(P)$ to be the number of symmetric matrices $B=(b_{ij})_{1\leq i,j\leq 5}\in M_{5\times 5}(\Z)$ with $|b_{ij}|\leq P$ for all $1\leq i\leq j\leq 5$, such that $S_{A,B}$ is smooth of codimension $2$ and $\Br_1(U)/ \Br_0(U)\neq 0$ for any hyperplane $H$ with $H\cap S_{A,B}$ geometrically irreducible.\par

\begin{theorem}\label{prop1}
Let $A\in M_{5\times 5}(\Z)$ and assume that $\det(A)\neq 0$. Then, for every $\eps >0$, one has the upper bound
$$N_{2}(P)\ll_{\eps,A} P^{14+1/5+\eps}.$$
\end{theorem}

In total, there are about $P^{15}$ integer matrices $B$ that we consider, and hence the exceptional ones, counted by $N_2(P)$, are sparse in this very precise sense. In particular, for any choice of $B$ outside of the exceptional set counted by $N_2(P)$, there is no algebraic Brauer-Manin obstruction to integral points for the system
$$Q_1(\bfx)=Q_2(\bfx)=0,\quad L(\bfx)=\pm1,$$
where $L(\bfx)$ is a sufficiently general linear form in the variables $\bfx$ with integral coefficients. It would be interesting to understand the transcendental part of the Brauer group of $U$ as well, see for example \cite{JS} for some discussions.\par
The frequency of violations of the local-global principle has recently also been studied by Mitankin \cite{Mitankin} for affine quadric surfaces. Even more is known for a few selected examples of families of projective varieties, see for example \cite{JS16} and \cite{BB2}.\par
Our strategy of proof for Theorem \ref{prop1} is inspired by work of Dietmann \cite{Die}. He bounds the number of monic integer polynomials of degree $n$ of bounded height which have a certain Galois group strictly smaller than $S_n$.\par
Under the assumptions of Theorem \ref{prop1}, the set of matrices $B$, counted by $N_{2}(P)$, is a thin subset in $15$-dimensional affine space (see Remark \ref{rk10}). Therefore, sieve methods, as used in Theorem 13.1 in \cite{Serre}, would lead to the bound
\begin{equation}\label{eqn1b}
N_2(P)\ll P^{14+1/2}\log P.
\end{equation}
With our methods, we improve upon the exponent in this estimate.\vspace{0.2cm}\par

{\bf Acknowledgements:} We would like to thank the referees for carefully reading this paper and for their suggestions that improved the presentation of the material.

\section{Preparations}
With two symmetric $5\times 5$ matrices $A$ and $B$, we associate the characteristic polynomial 
$$f(\lam,\mu;A,B):=\det (\lam A+\mu B),$$
which is homogeneous of degree $5$ in $\lam,\mu$. Sometimes, we write $f(\lam,\mu)$ for $f(\lam,\mu;A,B)$ when the matrices $A$ and $B$ are considered fixed. We recall that we can read off from the polynomial $f(\lam,\mu)$ whether the associated variety $S_{A,B}$ is smooth of codimension $2$.

\begin{lemma}[Proposition 3.26 in \cite{Wi}]
Let $A,B\in M_{5\times 5}(\Q)$ be two symmetric matrices and $S_{A,B}$ be defined as in (\ref{eqn1}), where $Q_1$ and $Q_2$ are the two quadratic forms associated with $A$ and $B$. Then $S_{A,B}$ is smooth over $\Q$ and pure of codimension $2$ as a variety in $\P^4$ if and only if the characteristic polynomial $f(\lam,\mu)$ is not identically zero and separable.
\end{lemma}

Let $\Disc (f)$ be the discriminant of $f(\lam,\mu)$ (see for example \cite{Schur}). If we fix $A\in M_{5\times 5}(\Z)$ then $\Disc(f)$ is a polynomial in the coefficients $(b_{ij})$ of $B$. Moreover, $S_{A,B}$ is smooth of codimension $2$ if and only if $\Disc (f)\neq 0$. Hence we can trivially handle the count of matrices $B$ of bounded height, for which $S_{A,B}$ is not a del Pezzo surface of degree $4$. Note that $\Disc(f)$ is not identically zero if we assume that $\det(A)\neq 0$. In this case, we have
\begin{equation}\label{eqn3}
\sharp\{B\in M_{5\times 5}(\Z): B^T=B, |b_{ij}|\leq P\, \forall 1\leq i,j\leq 5, \Disc(f)=0\}\ll P^{14}.
\end{equation}
Note that in this estimate the implied constant does not depend on the matrix~$A$.\par
The roots of the characteristic polynomial $f(\lam,\mu)$ correspond to pairs of pencils of conics contained in $S_{A,B}$. These pencils of conics generate $\Pic(U_{\bar \Q})$ up to index two and the algebraic part of the Brauer group of $U$ is determined by the operation of the Galois group on these conics. Moreover, one can already read off from the characteristic polynomial $f(\lam,\mu)$ some information about $2$-torsion elements contained in $\Br_1(U)/\Br_0(U)$. Note that, if $H\cap S_{A,B}$ is geometrically integral and $S_{A,B}$ is smooth of codimension $2$, then $U= X\setminus H$ is an open del Pezzo surface of degree four as in the language of \cite[Definition 2.5]{JS2}. As we work over the base field $\Q$, we have an isomorphism
$$\Br_1(U)/\Br_0(U)\cong H^1(\Gal(\bar \Q/\Q), \Pic(U_{\bar \Q})).$$
The latter group is analysed in detail in \cite{JS2} for an open degree four del Pezzo surface, and we now recall a result from that paper which forms one of the key inputs for our estimates for $N_2(P)$. 

\begin{theorem}[Corollary 3.11 in \cite{JS2}]\label{thm2.2}
Let $S_{A,B}$ be a smooth del Pezzo surface of degree $4$, as defined in (\ref{eqn1}), and $H\subset \P^4$ a hyperplane such that $H\cap S_{A,B}$ is geometrically integral. If $\Br_1(U)/\Br_0(U)\neq 0 $ then the characteristic polynomial $f(\lam,\mu)$ has a rational root or splits off a factor of degree $2$ over the rationals.
\end{theorem}

For the first part of our estimates to follow, we use the following result due to Bombieri and Pila \cite{BomPi}.

\begin{theorem}[Theorem 4 in \cite{BomPi}]\label{thm2.3}
Let $G(x,y)\in \Z[x,y]$ be an absolutely irreducible polynomial of absolute degree $d$. Let $N\geq \exp(d^6) $ be an integer. Then one has the bound
$$\sharp\{x,y\in( [0,N]\cap \Z)^2: G(x,y)=0\}\leq N^{1/d}\exp\left(12\sqrt{d\log N\log\log N}\right) .$$
\end{theorem}

Note that this result is more precise than what we need, as in our application the degree $d$ is fixed.
\section{Irreducibility results and a resolvent construction}

In our proof of Theorem \ref{prop1}, we need the following result on geometric irreducibility of fibers. It can be found in online lecture notes by B. Osserman \cite[Proposition 2.3]{Os}. For convenience of the reader, we give our own proof here.

\begin{proposition}\label{Osserman}
Let $A$ be a Noetherian ring, set $X:= \Spec \, A$ and let ${g\in A[x_1,\ldots, x_n]}$ be a polynomial of total degree $d$. Then 
there is a Zariski closed subset $Z\subset X$ with the following property.\par
A prime ideal $\grp \in X$ is contained in $Z$ if and only if there exists an algebraically closed field extension $k$ of the residue field $\kap(\grp)$ such that $g$ is reducible or has degree strictly less than $d$ when considered as a polynomial over $k$. 
\end{proposition}

\begin{remark}
Note that the last statement in Proposition \ref{Osserman} is equivalent to saying that for every algebraically closed extension field $k$ of $\kap(\grp)$ the polynomial $g$ is reducible or has degree less than $d$ when considered as a polynomial over $k$. (See for example \cite[Exercise II.3.15]{Hartshorne}.)
\end{remark}

\begin{proof}
{\em Step 1:} 
First, we consider the homogenization $G(x_0,\ldots, x_n)$ of the polynomial $g(x_1,\ldots, x_n)$, defined by
$$G(x_0,\ldots, x_n):= x_0^d g\left(\frac{x_1}{x_0},\ldots, \frac{x_n}{x_0}\right).$$
Then $G$ defines a hypersurface $V\subset \P_A^n$. Let $\grp\in X$ and $k$ some extension field of $\kap(\grp)$. We note that the polynomial $g$ is reducible or of degree strictly smaller than $d$ over $k$ if and only if $G$ is reducible over $k$. Moreover, $G$ is reducible over $k$ if and only if the reduction of $V$ over $k$ is reducible or not reduced over~$k$.\par
By \cite[Th\'eor\`eme 9.7.7]{EGAIV}, the set of prime ideals $\grp$ such that $G$ is reducible over $k$ is constructible. In order to show that this set is even Zariski closed, it is hence sufficient to show that it is closed under specialization. We will prove the following claim.\vspace{0.2cm}\\
\noindent
{\bf Claim 1:} Let $\grp\subset \grP$ be prime ideals in $A$. If the image of $g$ in $Q(A/\grp)[x_1,\ldots, x_n]$ is of degree strictly smaller than $d$ or becomes reducible over a finite extension field of $Q(A/\grp)$, then the image of $g$ in $Q(A/\grP)[x_1,\ldots, x_n]$ is reducible or of degree smaller than $d$ after taking a finite extension field. (Here, we write $Q(B)$ for the quotient field of an integral domain $B$.)\vspace{0.2cm}\\
Since we assumed $A$ to be Noetherian, it is enough to prove the claim for the case that there is no other prime ideal between $\grp$ and $\grP$.\vspace{0.2cm}\\
{\em Step 2:} If the reduction of $g$ modulo $\grp$ is of degree strictly smaller than $d$ then clearly the same is true for the reduction of $g$ modulo $\grP$. We hence assume that the image of $g$ in $Q(A/\grp)[x_1,\ldots , x_n]$ is reducible over $\overline{Q(A/\grp)}$.\vspace{0.2cm}\\
{\em Step 3:} 
In replacing the ring $A$ by $A/\grp$, we may assume w.l.o.g. that $A$ is an integral domain and $\grp=(0)$ and that $\grP$ is a minimal prime ideal. We now consider the localization $A_{\grP}$, which is a Noetherian, one-dimensional, local integral domain with maximal ideal $\grP A_{\grP}$. Moreover, $A_\grP/\grP A_\grP= Q(A/\grP)$. It is hence sufficient to establish the following claim.\vspace{0.2cm}\\
\noindent
{\bf Claim 2:} Let $A$ be a one-dimensional, Noetherian, local integral domain with maximal ideal $\grm$ and assume that $g$, considered as a polynomial in $Q(A)[x_1,\ldots, x_n]$, is reducible over a finite extension field. Then the reduction of $g$ in $(A/\grm)[x_1,\ldots, x_n]$ is of degree strictly smaller than $d$ or is reducible, possibly over a finite extension field of $A/\grm$.\vspace{0.2cm}\\
{\em Step 4:} Assume that $Q(A)\subset L$ is a finite extension field, such that $g$ factors over $L$. Let $R$ be the integral closure of $A$ in $L$. By the Going-up Theorem \cite[Theorem 9.4 i)]{Matsumura}, there is a prime ideal $\grn \subset R$ of height one lying above~$\grm$. We note that $R$ is Noetherian by the theorem of Krull-Akizuki \cite[Theorem 11.7]{Matsumura}. Hence, $R_\grn$ is a Noetherian, one-dimensional, local domain, which is integrally closed in its fraction field. By \cite[Theorem 11.2]{Matsumura}, $R_\grn$ is a discrete valuation ring. We have assumed that $g$ factors over $Q(R_\grn)=Q(R)=L$. Since $R_\grn$ is a discrete valuation ring, it is a unique factorization domain and Gauss's lemma implies that $g$ already factors over $R_\grn$.\vspace{0.2cm}\\
{\em Step 5:} As $g$ is reducible in the polynomial ring $R_\grn[x_1,\ldots, x_n]$, its reduction in $R_\grn/\grn R_\grn[x_1,\ldots, x_n]$ is also reducible or of strictly smaller degree. We conclude the proof in noting that $R_{\grn}/\grn R_\grn = Q(R/\grn)$ and that $Q(A/\grm) \subset Q(R/\grn)$ is a finite extension field. 
\end{proof}

In the following, we assume that both $A$ and $B$ are symmetric matrices and write $(b_{ij})_{1\leq i\leq j\leq 5}$ for the coefficients determining the matrix $B$. Moreover, we set $\bfb':=(b_{ij})_{1\leq i\leq j\leq 5, (i,j)\neq (1,1)}$ for the vector consisting of all entries except for the first one $b_{11}$.

\begin{lemma}\label{lemZ0}
Let $A=\Id_{5}$ be the identity matrix. Then there is a Zariski closed subset $Z_0\subsetneq \A_\Q^{14}$ with the following property. If $\bfb'\in \A^{14}(\Qbar)\setminus Z_{0}(\Qbar)$, then $f(\lam,1;\Id_{5},B)$ is absolutely irreducible of degree five as a polynomial in the variables $\lam$ and $b_{11}$.
\end{lemma}

\begin{proof}
We apply Proposition \ref{Osserman} to the ring $A=\Q[b_{12},\ldots, b_{55}]$ and the polynomial $f(\lam,1;\Id_5,B)\in A[b_{11},\lam]$. Let $Z_0\subset \A_\Q^{14}$ be the Zariski closed subset, given by Proposition \ref{Osserman}. It remains to show that $Z_0$ is not equal to the whole $14$-dimensional affine space. Calculations using {\tt magma} show that the polynomial
\begin{equation*}
g(\lam,b_{11},b_{12}):=\det \left(\lam I_5 + \begin{pmatrix}
    b_{11} & b_{12} &0 &0 &0\\ 
    b_{12} & b_{12}&b_{12}&0&0\\
    0&b_{12}&b_{12}&b_{12}&0\\
    0&0&b_{12}&b_{12}&b_{12}\\
    0&0&0&b_{12}&b_{12}
    \end{pmatrix} \right)
\end{equation*}
defines an irreducible curve of genus 0 in the projective plane. Moreover, up to height 100 we find eight regular $\Q$-rational points. Hence, the curve is geometrically irreducible and the polynomial $g(\lam,b_{11},1)$ is absolutely irreducible and of degree five. We conclude that $Z_0\subsetneq \A_\Q^{14}$ as desired.
\end{proof}

Since $f(\lam,1;\Id_5,B)$ is homogeneous of degree five, we obtain the following corollary. 

\begin{corollary}\label{irred2}
The polynomial $f(\lam,1;\Id_5,B)$ is absolutely irreducible in the $16$ variables $\lam$ and $(b_{ij})_{1\leq i\leq j\leq 5}$.\hspace{6cm}$\quad\quad\quad\quad\qed$
\end{corollary}

Now we show that similar statements hold when we replace the identity matrix $I_5$ by some symmetric invertible matrix $A\in M_{5\times 5}(\Q)$. 

\begin{lemma}\label{lem5} 
Let $A\in M_{5\times 5}(\Q)$ be a symmetric matrix and assume that $\det(A)\neq 0$. Then the polynomial $f(\lam,1;A,B)$ is absolutely irreducible in the $16$ variables $\lam$ and $(b_{ij})_{1\leq i\leq j\leq 5}$. 
\end{lemma}

\begin{proof}
As we ask for absolute irreducibility, we may work over $\bar{\Q}$. Let $T$ be an invertible matrix such that $T^tAT=\Id_5$. Then we have
\begin{equation}\label{eqn20}
\begin{split}
f(\lam,1;A,B)&= \det (\lam A+B)= \det (\lam (T^{-1})^t\Id_5T^{-1}+B)\\ &= \det(T^{-1})^2 \det (\lam I_5 + T^tBT).
\end{split}
\end{equation}
By Corollary \ref{irred2}, the polynomial $\det(\lam I_5 + B)$ is absolutely irreducible, considered as a polynomial in $16$ variables. Hence, the polynomial on the right hand side of (\ref{eqn20}) is absolutely irreducible in $\lam$ and $(b_{ij})_{1\leq i\leq j\leq 5}$, as the linear variable substitution $B \mapsto T^tBT$ does not affect absolute irreducibility.  
\end{proof}

For the case $A=\Id_5$, Lemma \ref{lemZ0} shows that the polynomial $f(\lam,1;\Id_{5},B)$ is absolutely irreducible in $\lam$ and $b_{11}$, for almost all $\bfb'$. We now provide a comparable statement for a general matrix $A$. For the proof, we use Lemma \ref{lem5} together with a Bertini-type theorem for absolute irreducibility to deduce that there is a curve in the corresponding family that is absolutely irreducible. This will suffice for another application of Proposition \ref{Osserman}.

\begin{proposition}\label{irred5}
Let $A\in M_{5\times 5} (\Q)$ be a symmetric matrix with $\det (A)\neq 0$. Then there exist a tuple $\bfv':=(v_{ij})_{1\leq i\leq j\leq 5, (i,j)\neq (1,1)}\in \Z^{14}$ and a Zariski closed subset $Z_A\subsetneq \A_\Q^{14}$ with the following property.\par
Define the polynomial 
\begin{equation*}
h_A(\lam, b_{11},\bfc'):= \det \left(\lam A + \begin{pmatrix}
    b_{11} & c_{12}+v_{12}b_{11} &\ldots &c_{15}+v_{15}b_{11}\\ 
    c_{12}+v_{12}b_{11} & c_{22}+v_{22}b_{11}&\ldots&c_{25}+v_{25}b_{11}\\
    \vdots&\ddots &&\vdots\\
    c_{15}+v_{15}b_{11}&c_{25}+v_{25}b_{11} &\ldots &c_{55}+v_{55}b_{11}
    \end{pmatrix} \right).
\end{equation*}
Then $h_A(\lam, b_{11},\bfc')$ is absolutely irreducible as a polynomial in $\lam$ and $b_{11}$, for each $\bfc':=(c_{ij})_{1\leq i\leq j\leq 5, (i,j)\neq (1,1)}\in \A^{14}(\bar \Q)\setminus Z_A(\bar \Q)$. 
\end{proposition}

\begin{proof}
By Lemma \ref{lem5}, the polynomial $f(\lam,1;A,B)$ is absolutely irreducible in the $16$ variables $\lam$ and $(b_{ij})_{1\leq i\leq j\leq 5}$. 
Consider the affine variety $X\subset \A_\Q^{16}$, given by $f(\lam,1;A,B)=0$, which is then geometrically integral. Define the projection
\begin{equation*}
\begin{split}
\phi_{12}: X&\rightarrow \A_{\Q}^2\\
(\lam,(b_{ij}))&\mapsto (b_{11},b_{12}),
\end{split}
\end{equation*}
and note that this is a dominant map.
Then an application of \cite[Th\'eor\`eme 6.3.3)\, and 4)]{Jou} shows that the intersection
\begin{equation*}
\begin{split}
X^{(2)}: \quad\quad\quad\quad f(\lam,1;A,B)&=0,\\ u_{12}+v_{12}b_{11}+w_{12}b_{12}&=0
\end{split}
\end{equation*}
is geometrically integral for almost all triples $(u_{12},v_{12},w_{12})\in \A^3$ in the sense that the exceptional set is contained in a Zariski closed subset of $\A^3$. By homogeneity, the same holds after normalizing $w_{12}=-1$, say. Hence, there are integers $u_{12}$ and $v_{12}$ such that $X^{(2)}$ is geometrically integral. We now apply the same argument to 
to find integers $u_{13}$ and $v_{13}$ such that the intersection $$X^{(3)}=X^{(2)}\cap \{b_{13}=u_{13}+v_{13}b_{11}\}$$ is geometrically integral. We repeat this process in total $14$ times to obtain tuples $(u_{ij})_{1\leq i\leq j\leq 5, (i,j)\neq (1,1)}, (v_{ij})_{1\leq i\leq j\leq 5, (i,j)\neq (1,1)}\in \Z^{14}$ such that the affine curve given by
\begin{equation*}
\det \left(\lam A + \begin{pmatrix}
    b_{11} & u_{12}+v_{12}b_{11} &\ldots &u_{15}+v_{15}b_{11}\\ 
    u_{12}+v_{12}b_{11} & u_{22}+v_{22}b_{11}&\ldots&u_{25}+v_{25}b_{11}\\
    \vdots&\ddots &&\vdots\\
    u_{15}+v_{15}b_{11}&\ldots &\ldots &u_{55}+v_{55}b_{11}
    \end{pmatrix} \right)=0
\end{equation*}
is geometrically integral. (Note that the maps $\phi_{1j}$ are all dominant since the polynomial $f(\lam,1;A,B)$ contains the term $\det (A) \lam^5$ and one can therefore always solve for $\lam$.) Hence, $h_A(\lam, b_{11},\bfu)$ is absolutely irreducible as a polynomial in $\lam$ and $b_{11}$. An application of Proposition \ref{Osserman} provides us with a Zariski closed subset $Z_A\subsetneq \A_\Q^{14}$ with the desired property. 
\end{proof}

The results so far are enough to deal with the case of Brauer classes that can occur due to $f(\lam,\mu;A,B)$ having a rational root in $\lam,\mu$. We now need to provide similar irreducibility results for a resolvent that detects when the polynomial $f(\lam,\mu;A,B)$ splits off a factor of degree two. Recall that we have defined $f(\lam,\mu;A,B)= \det (\lam A+\mu B)$. We can write $f(\lam,\mu)$ in the form
$$f(\lam,\mu)=\sum_{l=0}^5p_l(A,B)\mu^{5-l}\lam^l,$$
with $p_l(A,B)$ a polynomial of degree $l$ in the coefficients of the matrix $A$ and of degree $5-l$ in the coefficients of the matrix $B$. Moreover, $p_0(A,B)=\det(B)$ and $p_5(A,B)=\det(A)$. Assume that $\det(A)\neq 0$ and that $A$ and $B$ are fixed. We rewrite $f(\lam,\mu)$ as 
\begin{equation}\label{eqn22}
f(\lam,\mu)= \det(A) \sum_{l=0}^5 \frac{p_l(A,B)}{\det(A)}\mu^{5-l}\lam^l.
\end{equation}
Over some algebraic closure of $\Q$, we then have 
\begin{equation}\label{eqn23}
f(\lam,\mu)= \det(A) \prod_{i=1}^5 (\lam +\alp_i\mu).
\end{equation}
We now define the resolvent 
$$\Phi(z;A,B):=\det (A)^4 \prod_{1\leq i<j\leq 5}\!\!\! (z+\alp_i+\alp_j).$$

\begin{lemma}
{\rm a)} The resolvent $\Phi(z;A,B)$ has the form
$$\Phi(z;A,B) = \sum_{k=0}^{10} q_k(A,B) z^k,$$
with $q_k(A,B)\in \Z[A,B]$ polynomials in the coefficients of the two matrices $A,B$. \\
{\rm b)} The leading coefficient satisfies $q_{10}(A,B)=\det( A)^4$ and, for $0\leq k\leq 10$, the polynomial $q_k(A,B)$ is homogeneous of degree $10-k$ in the coefficients of~$B$ and homogeneous of degree $10+k$ in the coefficients of $A$.\\
{\rm c)} If $f(\lam,\mu;A,B)$ splits off a factor of degree two over the rationals then $\Phi(z;A,B)$ has a rational root $z\in \Q$. 
\end{lemma}

\begin{proof}
Using the notation above, we can first write
$$\Phi(z;A,B)= \det (A)^4 \sum_{k=0}^{10} \tilde{q}_k(\bfalp)z^k,$$
with $\tilde{q}_k(\bfalp)$ polynomials in the roots $\alp_i$, for $1\leq i\leq 5$. Note that $\tilde{q}_k(\bfalp)$ are symmetric polynomials in the variables $\alp_i$. Hence, they can be expressed in terms of the following elementary symmetric polynomials
\begin{equation*}
\begin{split}
 E_1:=\sum_{i=1}^5 \alp_i,\quad E_2:=\sum_{i<j}\alp_i\alp_j,\quad E_3:= \sum_{i<j<k}\alp_i\alp_j\alp_k,\\ E_4:=\sum_{i<j<k<l} \alp_i\alp_j\alp_k\alp_l,\quad E_5:=\prod_{i=1}^5 \alp_i.
\end{split}
\end{equation*}
We compare equation (\ref{eqn22}) with equation (\ref{eqn23}) and obtain 
$$ \det(A) \sum_{l=0}^5 \frac{p_l(A,B)}{\det(A)}\mu^{5-l}\lam^l= \det(A) \prod_{i=1}^5 (\lam +\alp_i\mu).$$
Hence 
$$E_i=\frac{p_{5-i}(A,B)}{\det(A)},\quad\mbox{for } 1\leq i\leq 5.$$
By a short calculation in {\tt magma}, we can express each of the polynomials $\tilde{q}_k(\bfalp)$ in the elementary symmetric polynomials $E_i$, for $1\leq i\leq 5$. We find that they are of total degree at most four in this representation. Consequently, all the
$$q_k(A,B):= \det(A)^4 \tilde{q}_k(\bfalp),\quad 0\leq k\leq 10,$$
are polynomials in $\Z[A,B]$, where $q_k(A,B)$ is homogeneous of degree $10-k$ in the coefficients of $B$ and homogeneous of degree $10+k$ in the coefficients of~$A$. This completes part a) and b) of the proof of the lemma. Part c) follows directly from the construction of the resolvent polynomial $\Phi(z;A,B)$.
\end{proof}

We next need an statement analogous to Lemma \ref{irred5}, for the resolvent $\Phi(z;A,B)$. The strategy of proof is the same as that for the polynomial $f(\lam,\mu)$ itself.

\begin{proposition}\label{irred6}
Let $A\in M_{5\times 5}(\Q)$ be a symmetric matrix with $\det(A)\neq 0$. Then there exist a tuple $\bft':=(t_{ij})_{1\leq i\leq j\leq 5, (i,j)\neq (1,1)}\in \Z^{14}$ and a Zariski closed subset $V_A\subsetneq \A_\Q^{14}$ with the following property.\par
The expression $\Phi(z;A,(b_{11},\bfc'+b_{11}\bft'))$
is absolutely irreducible as a polynomial in the variables $z$ and $b_{11}$ if $\bfc':=(c_{ij})_{1\leq i\leq j\leq 5, (i,j)\neq (1,1)}\in \A^{14}(\bar \Q)\setminus V_A(\bar \Q)$. 
\end{proposition}

\begin{proof}
First, a short calculation using {\tt magma} shows that the polynomial $$\Phi(z;\Id_5,(b_{11},\bfc'+b_{11}\bft'))$$ is absolutely irreducible in $z,b_{11}$, if we set $\bft':=\mathbf{0}$ and 
$c_{ij}:=1$, for ${|i-j|\leq1}$, and $c_{ij}:=0$, otherwise. Indeed, in this case, the homogenization of the polynomial ${\Phi(z;\Id_5,(b_{11},\bfc'+b_{11}\bft'))}$ defines an irreducible curve of genus $3$ with two regular rational points of height less than $100$.\par
Since $\Phi(z;\Id_5,(b_{11},\bfc'+b_{11}\bft'))$ is absolutely irreducible of degree $10$ in $z$ and $b_{11}$ for these special choices of $\bfc'$ and $\bft'$, we deduce that $\Phi(z;\Id_5,B)$ is absolutely irreducible in the $16$ variables $z$, $(b_{ij})_{1\leq i\leq j\leq 5}$.\par
Now let $A\in M_{5\times 5}(\Q)$ be an arbitrary symmetric matrix with $\det(A)\neq 0$. Over $\bar \Q$, we can diagonalize $A$ with an invertible matrix $C$ such that $$C^tAC=\Id_5.$$ Then we have
\begin{equation*}
\begin{split}
f(\lam,\mu;A,B)&=\det(C)^{-2}\det (\lam C^tAC+\mu C^tBC)\\
&= \det(C)^{-2} f(\lam,\mu;\Id_5,C^tBC).
\end{split}
\end{equation*}
For the corresponding resolvents, we hence get the relation
$$\Phi(z;A,B)= \det(C)^{-8} \Phi(z;\Id_5,C^tBC).$$
We conclude that $\Phi(z;A,B)$ is an absolutely irreducible polynomial in the $16$ variables $z,(b_{ij})_{1\leq i\leq j\leq 5}$. The rest of the proof proceeds in exactly the same way as the proof of Lemma \ref{irred5}. Note that, for $A$ fixed with $\det(A)\neq 0$, the polynomial $\Phi(z;A,B)$ always has degree $10$. 
\end{proof}

\section{Proof of Theorem \ref{prop1}}

In this section, we prove Theorem \ref{prop1}. Let $A\in M_{5\times 5}(\Z)$ be fixed and assume that $\det(A)\neq 0$. Let $S^{(1)}(P)$ be the number of symmetric matrices $B\in M_{5\times 5}(\Z)$ with $|b_{ij}|\leq P$, for all $1\leq i\leq j\leq 5$, such that $f(\lam,\mu;A,B)$ has a rational root in $\lam,\mu$. Similarly, let $S^{(2)}(P)$ be the number of such matrices $B$ such that $\Phi(z;A,B)$ has a rational root in $z$. 

By Theorem \ref{thm2.2}, and the construction of the resolvent polynomial $\Phi(z;A,B)$, we can bound $N_2(P)$ by 
\begin{equation}\label{eqn13}
N_2(P)\ll S^{(1)}(P) + S^{(2)}(P).
\end{equation}

If $f(\lam,\mu;A,B)$ has a rational root in $\lam,\mu$, then, by homogeneity, it also has a solution $\lam,\mu\in \Z^2$ with $\gcd(\lam,\mu)=1$. Since the coefficient of $\lam^5$ in $f(\lam,\mu)$ is equal to $\det(A)$, we deduce that $\mu | \det(A)$. For a divisor $d|\det (A)$, we define $S_d^{(1)}(P)$ to be the number of symmetric matrices $B\in M_{5\times 5}(\Z)$ with coefficients bounded by $P$ such that $f(\lam,d;A,B)$ has an integer root in $\lam$. Then we can bound
\begin{equation}\label{eqn14}
S^{(1)}(P) \leq \sum_{d|\det (A)}S_d^{(1)}(P),
\end{equation}
where the sum over $d$ runs over all positive and negative divisors of $\det (A)$.\par
Similarly, if $\Phi(z;A,B)$ has a rational root, say $\smash{z=\frac{r}{s}}$, with $r,s\in \Z^2$ and $\gcd(r,s)=1$ then $\smash{s|\det(A)^4}$. For any divisor $\smash{d|\det(A)^4}$, we let $\smash{S_d^{(2)}(P)}$ be the number of symmetric  matrices $B\in M_{5\times 5}(\Z)$ with coefficients bounded by $P$ such that $d^{10} \Phi(d^{-1}z;A,B)$ (which is again a polynomial with integer coefficients) has an integer root in $z$. Hence, we have the upper bound
$$S^{(2)}(P)\leq \sum_{d|\det(A)^4} S_d^{(2)}(P),$$
where again the sum is over all positive and negative divisors $d$ of $\det(A)^4$.\par
With $d|\det(A)$ fixed, we now proceed to bound $S_d^{(1)}(P)$. We recall the notation
$$f(\lam,\mu;A,B)=\sum_{l=0}^5 p_l(A,B)\mu^{5-l}\lam^l,$$
with $p_l(A,B)$ polynomials of degree $5-l$ in the coefficients of $B$. If $B$ is a matrix with coefficients bounded by $P$ then $p_l(A,B)\ll P^{5-l}$. As soon as $f(\lam,d;A,B)$ has a complex root $\lam_0$, then we deduce that there is a positive constant $C_0\geq 1$ such that $|\lam_0|\leq C_0 P$, by Theorem (27,3) in \cite{Ma} (see also Lemma 1 in \cite{Die}). 
Let $\bfv'$ and $Z_A$ be as in Lemma \ref{irred5}. Then $f(\lam,1;A,(b_{11},\bfc'+b_{11}\bfv'))$ 
is absolutely irreducible in $\lam,b_{11}$ if $\bfc'\notin Z_A$. 
By a linear change of variables, we deduce that 
$$f(\lam,d;A,(b_{11},\bfc'+b_{11}\bfv'))=f(\lam,1;A, (db_{11},d\bfc'+db_{11}\bfv'))$$
is absolutely irreducible in the two variables $\lam,b_{11}$ if $\bfc'\notin d^{-1}Z_A$. 
Let $R_d^{(1)}(P)$ be the number of vectors $(b_{11},\bfc')\in (\Z\cap [-P,P])^{15}$ such that the polynomial $f(\lam,d;A,(b_{11},\bfc'+b_{11}\bfv'))$ has an integer root in $\lam$. Then there is a positive constant $C_1$ such that
\begin{equation}\label{eqn17}
S_d^{(1)}(P)\leq R_d^{(1)}(C_1 P).
\end{equation}
For $\bfc'\in \Z^{14}$, we define
$$T_d^{(1)}(\bfc';P) := \sharp\{(b_{11},\lam)\in (\Z\cap[-P,P])^2:f(\lam,d;A,(b_{11},\bfc'+b_{11}\bfv'))=0\}.$$
Let
$$E_d^{(1)}(P):=\sharp\{\bfc'\in (\Z\cap [-P,P])^{14}: \bfc'\in d^{-1}Z_A\}.$$
Then we can bound $R_d^{(1)}(C_1P)$ by
$$R_d^{(1)}(C_1P)\ll \sum_{\substack{\bfc'\in (\Z\cap [-C_1P,C_1P])^{14}\\ \bfc' \notin d^{-1}Z_A}}\!\!\!\!\!\!\!\!\! T_d^{(1)}(\bfc'; C_2P)+PE_d^{(1)}(C_1P),$$
for a sufficiently large constant $C_2$. 
An application of Theorem \ref{thm2.3} gives the bound
$$T_d^{(1)}(\bfc'; C_2P) \ll_\eps P^{1/5+\eps},$$
with an implied constant that is independent of $\bfc'$. On the other hand, as $Z_A\subsetneq \A_\Q^{14}$ is Zariski closed, we have
$$E_d^{(1)}(C_2 P)\ll P^{13}.$$
Together these estimates lead to the bound
$$R_d^{(1)}(C_1P)\ll_\eps P^{14+1/5+\eps},$$
for any $\eps >0$. Here, the implied constant may depend on $d$ and $C_1$, but is independent of $P$. Since $A$ is considered fixed, we can trivially perform the summation in equation (\ref{eqn14}), and together with equation (\ref{eqn17}), we obtain the bound
$$S^{(1)}(P)\ll_\eps P^{14+1/5+\eps}.$$
The bound for $S^{(2)}(P)$ is obtained in the same way as the bound for $S^{(1)}(P)$, with Proposition \ref{irred6} in place of Proposition \ref{irred5}. In fact, since $\Phi(z;A,B)$ is typically irreducible of degree $10$, one obtains the bound
$$S^{(2)}(P)\ll_\eps P^{14+1/10+\eps},$$
for any $\eps >0$.

\begin{remark}\label{rk10}
Let $A$ be as above. As the polynomials $f(\lam,1;A,B)$ and $\Phi(z;A,B)$ are irreducible in $\bar\Q(b_{ij})_{1\leq i\leq j\leq 5}[\lam]$ and $\bar\Q(b_{ij})_{1\leq i\leq j\leq 5}[z]$, respectively, the polynomial interpretation of thin sets as in section 9.1 of \cite{Serre} shows that the matrices $B$ counted by $N_2(P)$ indeed form a thin subset. 

\end{remark}

\bibliographystyle{amsbracket}

\begin{thebibliography}{EGAIV}

\bibitem[Be]{BerA17}
J. Berg, \emph{Obstructions to integral points on affine Ch{\^a}telet surfaces}, {\tt arXiv:1710.07969}.

\bibitem[BSD]{BSD75}
B. J. Birch and Sir Peter Swinnerton-Dyer, \emph{The Hasse problem for rational surfaces}, in: Collection of articles dedicated to Helmut Hasse on his seventy-fifth birthday~III,  J.\ Reine Angew.\ Math. {\bf 274/275} (1975), 164--174.

\bibitem[BP]{BomPi}
E. Bombieri and J. Pila, \emph{The number of integral points on arcs and ovals}, Duke Math. J. {\bf 59} (1989), 337--357.

\bibitem[BCP]{BCP}
W.\ Bosma, J.\ Cannon, and C.\ Playoust, {\em The Magma
algebra system~I. The user language,} J.\ Symbolic Comput.\ {\bf 24} (1997), 235--265.

\bibitem[BB2]{BB2}
R. de la Bret\`eche and T. D. Browning, \emph{Density of Ch\^{a}telet surfaces failing the Hasse principle}, Proc. London Math. Soc. {\bf 108} (2014), 1030--1078.

\bibitem[BL]{BLA17}
M. Bright and J. Lyczak, \emph{A uniform bound on the Brauer groups of certain log K3 surfaces}, {\tt arXiv:1705.04529}. 

\bibitem[CW]{CTWit12}
J.-L. Colliot-Th\'el\`ene and O. Wittenberg, \emph{Groupe de Brauer et points entiers de deux familles de surfaces cubiques affines}, Amer.\ J.\ Math. {\bf 134} (2012), 1303--1327.

\bibitem[CX]{CX09}
J.-L. Colliot-Th\'el\`ene and F. Xu, \emph{Brauer-Manin obstruction for integral points of homogeneous spaces and representation by integral quadratic forms}, With an appendix by D.\ Wei and Xu, Compos.\ Math. {\bf 145} (2009), 309--363.

\bibitem[Di]{Die}
R. Dietmann, \emph{On the distribution of Galois groups}, Mathematika {\bf 58} (2012), 35--44.

\bibitem[EGAIV]{EGAIV}
A. Grothendieck; J. Dieudonn\'e (1964), {\em\'El\'ements de g\'eom\'etrie alg\'ebrique IV. \smash{\'Etude} locale des sch\'emas et des morphismes de sch\'emas}, 20, 24, 28, 32, Publications Math\'ematiques de l'IH\'ES. 20.

\bibitem[Ha]{Hartshorne}
R. Hartshorne, \emph{Algebraic Geometry}, Graduate Texts in Mathematics~52, {\em Springer,} New York 1977.

\bibitem[JS16]{JS16}
J. Jahnel and D. Schindler, \emph{On the number of certain Del Pezzo surfaces of degree four violating the Hasse principle}, Journal of Number Theory {\bf 162} (2016), 224--254.

\bibitem[JS1]{JS}
J. Jahnel and D. Schindler, \emph{On integral points on degree four del Pezzo surfaces},  Israel Journal of Mathematics {\bf 222} (2017), 21--62.

\bibitem[JS2]{JS2}
J. Jahnel and D. Schindler, \emph{On the algebraic Brauer classes on open degree four del Pezzo surfaces}, {\tt arXiv:1712.04745}.

\bibitem[Jo]{Jou}
J.-P. Jouanolou, \emph{Th\'eor\`emes de Bertini et applications}, Progress in Mathematics, 42, Birkh{\"a}user Boston, Boston, 1983.

\bibitem[Ma]{Ma}
M. Marden, \emph{Geometry of polynomials}, Second edition, Mathematical Surveys, No.~3, American Mathematical Society 1966. 

\bibitem[Mat]{Matsumura}
H. Matsumura, \emph{Commutative Ring Theory}, Cambridge University Press 1987.

\bibitem[Mi]{Mitankin}
V. Mitankin, \emph{Failures of the integral Hasse principle for affine quadric surfaces}, J. Lond. Math. Soc. (2)  {\bf 95}  (2017),  no. 3, 1035--1052.

\bibitem[Os]{Os}
B. Osserman, Lecture notes on \emph{Properties of fibers and applications},
\url{https://www.math.ucdavis.edu/~osserman/classes/248B-W12/notes/fibers.pdf}, 23rd May 2017. 

\bibitem[Sch]{Schur}
I. Schur, \emph{Vorlesungen {\"u}ber Invariantentheorie}, Springer, Die Grund\-lehren der mathematischen Wissenschaften 143, Berlin, 1968.


\bibitem[Se]{Serre}
J.-P. Serre, \emph{Lectures on the Mordell-Weil Theorem}, third edition, Aspects of Mathematics, Volume 15, {\em Springer}, Wiesbaden 1997.

\bibitem[Wi]{Wi}
O. Wittenberg, \emph{Intersections de deux quadriques et pinceaux de courbes de
genre~$1$}, Lecture Notes in Mathematics~1901, {\em Springer,} Berlin~2007.

\end{thebibliography}
\providecommand{\bysame}{\leavevmode\hbox to3em{\hrulefill}\thinspace}

\setlength\parindent{0mm}

\end{document}